\documentclass[a4paper,11pt]{article}

\usepackage{a4wide}
\usepackage[english]{babel}
\usepackage{amssymb}
\usepackage{amsmath}
\usepackage{amsfonts}
\usepackage{amsthm}
\usepackage{hyperref}
\hypersetup{colorlinks,citecolor=red,filecolor=purple,linkcolor=blue,urlcolor=black}
\usepackage{graphicx}
\usepackage{textcomp}
\usepackage{fourier-orns}
\usepackage{url}
\usepackage{enumerate}
\usepackage{graphicx}
\usepackage{textcomp}
\usepackage{url}
\usepackage{color}

\newcommand{\Nat}{\mathbb N}
\newcommand{\R}{\mathbb R}
\newcommand{\pse}{\psi'({\eta})}
\newcommand{\Z}{\mathbb Z}
\newcommand{\Esp}{\mathbb E}
\newcommand{\p}{\mathbb P}
\newcommand{\EE}[1]{\Esp\left[#1\right]}
\newcommand{\1}[1]{\mathbf{1}\!_{\left\{#1\right\}}}
\newcommand{\ext}{\mathrm{Ext}}
\newcommand{\as}{\quad\mathrm{ a.s.}}
\newcommand{\eps}{\varepsilon}

\newcommand{\intpos}{\int_0^\infty}

\newcommand{\peto}{\mathbb P^\star}
\newcommand{\dif}{\mathrm{d}}

\newcommand{\linf}[1]{\underset{#1\to\infty}\longrightarrow}

\newtheorem{prop}{Proposition}[section]

\newtheorem{lem}[prop]{Lemma}
\newtheorem{thm}[prop]{Theorem}
\newtheorem{rem}[prop]{Remark}
\newtheorem{cor}[prop]{Corollary}

\newtheoremstyle{main}{\topsep}{\topsep}%
{\itshape}% Body font
{}% Indent amount (empty = no indent, \parindent = para indent)
{\normalfont}% Thm head font
{\textcolor{red}{:}}% Punctuation after thm head
{ }% Space after thm head (\newline = linebreak)
{\danger\textbf{\textcolor{red}{\thmname{#1}\thmnumber{~#2}\thmnote{ (\normalfont #3)}}}}% Thm head spec

\theoremstyle{main}

\author{\textsc{Mathieu Richard$^{1}$}}
\title{Splitting trees with neutral mutations at birth}

\begin{document}
\maketitle
%\tableofcontents
\footnotetext[1]{CMAP, Ecole Polytechnique, Route de Saclay, 91128 Palaiseau Cedex, France. \href{mailto:mathieu.richard@cmap.polytechnique.fr}{\nolinkurl{mathieu.richard@cmap.polytechnique.fr}}}

\begin{abstract}
We consider a population model where individuals behave independently from each other and whose genealogy is described by a chronological tree called splitting tree. The individuals have i.i.d. (non-exponential) lifetime durations and give birth at constant rate to clonal or mutant children in an infinitely many alleles model with neutral mutations.

First, to study the allelic partition of the population, we are interested in its frequency spectrum, which, at a fixed time, describes the number of alleles carried by a given number of individuals and with a given age. We compute the expected value of this spectrum and obtain some almost sure convergence results thanks to classical properties of Crump-Mode-Jagers (CMJ) processes counted by random characteristics.

Then, by using multitype CMJ-processes, we get asymptotic properties about the number of alleles that have undergone a fixed number of mutations with respect to the ancestral allele of the population.
\end{abstract}

\section{Introduction}
We consider a general branching model with neutral mutations occurring at birth. We suppose that individuals carry alleles, have i.i.d. (and not necessarily exponential) life lengths and give birth at constant rate $b$ during their lives to children who can be mutants with probability $p$ or clones of their parents with probability $1-p$. We are working with an infinite alleles model, that is, when a mutation occurs, the allele of the mutant child was never encountered before. Moreover, mutations are neutral because they do not imply advantages or disadvantages (all individuals have identically distributed dynamics). %We call a \emph{family} the set of all individuals sharing a same allele.

Without mutation, the model is linked to a genealogical tree called \emph{splitting tree} \cite{Geiger1996,Geiger1997,Amaury_contour_splitting_trees}. Moreover, if $\Xi(t)$ denotes the number of alive individuals at time $t$, the process $\Xi:=(\Xi(t),t\geq0)$ is a \emph{Crump-Mode-Jagers process} (or general branching process) \cite{Jagers_BP_with_bio} which is binary (births occur singly) and homogeneous (constant birth rate).
%We are interested in asymptotic results as $t\rightarrow\infty$, so we suppose that $\Xi$ is supercritical and we work conditionally on its survival event.
\\

We are first interested in the allelic partition of the population and more precisely in properties of the \emph{frequency spectrum} $(M_t^{i,a},i\geq1)$ where $M_t^{i,a}$ is the number of distinct alleles carried by exactly $i$ individuals at time $t$ and that appeared after time $t-a$ (or equivalently the alleles whose ages are less than $a$ at time $t$). Roughly speaking, this spectrum classifies the different alleles depending on their ages and on their sizes. 

This kind of question was first studied by Ewens \cite{Ewens1972} who discovered the well known 'sampling formula' named after him and which describes the law of the allelic partition for a Wright-Fisher model with neutral mutations. In our model, we cannot get a counterpart of the Ewens' sampling formula in the sense that we cannot compute the joint law of $(M_t^{i,a},i\geq1)$ for fixed $0<a\leq t$. However, we obtain two kinds of results concerning the frequency spectrum $(M_t^{i,a},i\geq1)$. We first compute the expected value of the number of alleles carried by $i$ individuals at time $t$ and with age $a$. When the process $\Xi$ is supercritical, we also obtain the asymptotic behaviors of the frequency spectrum and of the relative abundances of alleles as $t\to\infty$ on the survival event of $\Xi$.

Similar models to ours have been studied in the literature.
In \cite{Griffiths1988a}, Griffiths and Pakes study the case of a Bienaym\'e-Galton-Watson process where children can independently be mutants with a given probability: the authors obtained asymptotic results about the number of alleles and the frequency spectrum at generation $n$ as $n\rightarrow\infty$. In \cite{Pakes1989}, Pakes gets analogous properties concerning continuous-time Markov branching processes. In particular, his formula of the expected frequency spectrum can be seen as a counterpart of ours, stated in Section \ref{moyenne spectre allelique}. These two works \cite{Griffiths1988a,Pakes1989} have recently been used by Kimmel and coworkers in several articles.  In \cite{Kimmel2010}, the authors are interested in the evolution of parts of DNA called Alu sequences. They model the evolution of these sequences using the infinite-alleles branching process of \cite{Griffiths1988a} with a linear-fractional offspring distribution. %Their results suggest that Alu sequences do not evolve neutrally.
In \cite{Kimmel2013}, the authors consider the model of Pakes \cite{Pakes1989} and they especially obtain an explicit expression of the limiting mean frequency spectrum in the particular case of birth and death processes and they also study the variance frequency spectrum.

In the two articles \cite{Bertoin2009,Bertoin2010}, Bertoin considers an infinite alleles model with neutral mutations in a subcritical or critical Bienaym\'e-Galton-Watson process where individuals independently give birth to a random number of clonal and mutant children according to the same joint distribution. In \cite{Bertoin2009}, he defines a tree of alleles where all individuals of a common allele are gathered in clusters and specifies the law of the allelic partition of the total population by describing the joint law of the sizes of the clusters and the numbers of their mutant children. %A tree of alleles is also built by Ta\"ib in \cite{Taib1992} for CMJ-processes.
In \cite{Bertoin2010}, Bertoin obtains the convergence of the sizes of allelic families for a large initial population and a small mutation rate.

Recent results have been obtained about splitting trees with mutations. In \cite{Cecile2013_II}, Delaporte studies sequences of splitting trees with general and neutral mutations occurring at birth and investigates scaling limits in a regime of large population size and rare mutations.
In two related papers \cite{Champagnat2010,Champ_Lamb_2}, Champagnat and Lambert consider a model really close to ours: the authors also work with splitting trees but in their model, individuals independently experience mutations at Poisson times during their lives. In \cite{Champagnat2010}, explicit formulas about the expected frequency spectrum at time $t$ are stated. In \cite{Champ_Lamb_2}, the authors are interested in large and old families; they obtain asymptotic results about the sizes of the largest families and about the ages of the oldest ones as $t\rightarrow\infty$. Their model with Poissonian mutations and our model with mutations 
occurring at birth are compared in \cite{Champ_Lamb_Rich} in the particular case of exponential lifelengths. 

 Finally, Ta\"ib \cite{Taib1992} considers CMJ-processes with a more general kind of mutations at birth (for example, the probability that a mutation occurs can depend on the age of the mother)
 and thanks to the theory of CMJ-processes counted with characteristics, he obtains several asymptotic results, especially about the frequency spectrum of the population.
However, in our particular case, some of the non-explicit limits he obtained can be computed thanks to a recent paper of Lambert \cite{Amaury_contour_splitting_trees} giving for example the one-dimensional marginals of $\Xi$ and its asymptotic behavior.
\\

We then obtain properties about the number of mutations undergone by alleles. More specifically, for $i\geq0$ and $t\geq0$, we study $L_i(t)$ the number of alleles at time $t$ that have been affected by $i$ more mutations than the ancestral allele and $K_i(t)$ the number of individuals that carry such alleles.
Our model can represent the spread of an epidemic: individuals are infected hosts, deaths are recoveries or actual deaths and births are transmissions of the disease, which can mutate to new strains. Then, $L^i(t)$ is simply the number of strains of the disease that are present at $t$ and that are at a distance $i$ from the original strain of the disease. Moreover, the process $K_i$ describes the number of individuals infected by such strains of the disease.

We compute the expected values of $L_i(t)$ and $K_i(t)$ and obtain asymptotic results about $K_i(t)$  
as $t\to\infty$
by considering a multitype CMJ-process where the type of an allele is the number of mutations it has undergone. Multitype branching processes are also used in carcinogenesis, that is, in the evolution of cancerous cells. In \cite{Durrett_Ovarian,Durrett201042}, cancerous cells are modeled by a multitype branching process where a cell is of type $k$ if it has undergone $k$ mutations and where the more a cell has undergone mutations, the faster it grows. The object of study is the time $\tau_k$ of appearance of the first cell of type $k$.
Branching processes and birth and death processes appear in other works dealing with the evolution of cancerous cells (see for example \cite{Nowak09122003,IwasaApril2006,Sagitov2009,Sagitov2013} and references therein).
For instance, in \cite{Sagitov2009}, the authors study the arrival time of the first resistant cell  in a model of cancerous cells undergoing a medical treatment and becoming resistant after having experienced a certain number of mutations.
\\

The paper is organized as follows: in Section \ref{preliminaires}, we expose the model that we consider and give some of its properties that will be useful to get our main results. Section \ref{section_spectre_frequence} is devoted to the study of the frequency spectrum and in Section \ref{nombre_de_mutations_subies}, we are interested in properties of the number of mutations undergone by alleles. 

\section{Preliminaries}\label{preliminaires}
\subsection{The model}
In this paper, as a population model, we consider genealogical trees satisfying the branching property and called \emph{splitting trees} \cite{Geiger1996,Geiger1997,Amaury_contour_splitting_trees}. They are random trees satisfying the following assumptions.
\begin{itemize}
\item At time $t=0$, there is only one progenitor.
\item All individuals have i.i.d. lifespans and reproduction behaviors.
\item Conditional on her birth date $\alpha$ and her lifespan $\zeta$, each individual gives birth to children at constant rate $b\in(0,\infty)$ during $(\alpha,\alpha+\zeta]$.
\item Births occur singly.
\end{itemize}
We denote the common lifespan distribution by $\Lambda(\cdot)/b$ where $\Lambda$ is a finite positive measure on $(0,+\infty]$ with total mass $b$ and called \emph{lifespan measure} \cite{Amaury_contour_splitting_trees}.

To this splitting tree, we add neutral mutations occurring at birth in the following way. We fix $p\in(0,1)$. When a birth event occurs, independently of others individuals, with probability $1-p$, a child is 
a clone of her mother, that is carries the same allele and a mutant with probability $p$. Moreover, when a mutant appears in the population, its allele was never carried before by any other individual. 
Thus, we consider an \emph{infinitely many alleles} model with \emph{neutral mutations} because all individuals behave in the same way regardless of allele.
On Figure \ref{figure_spectre des frequences}, one can find an illustration of this model.\\

\begin{figure}[!ht]
\begin{center}
\includegraphics{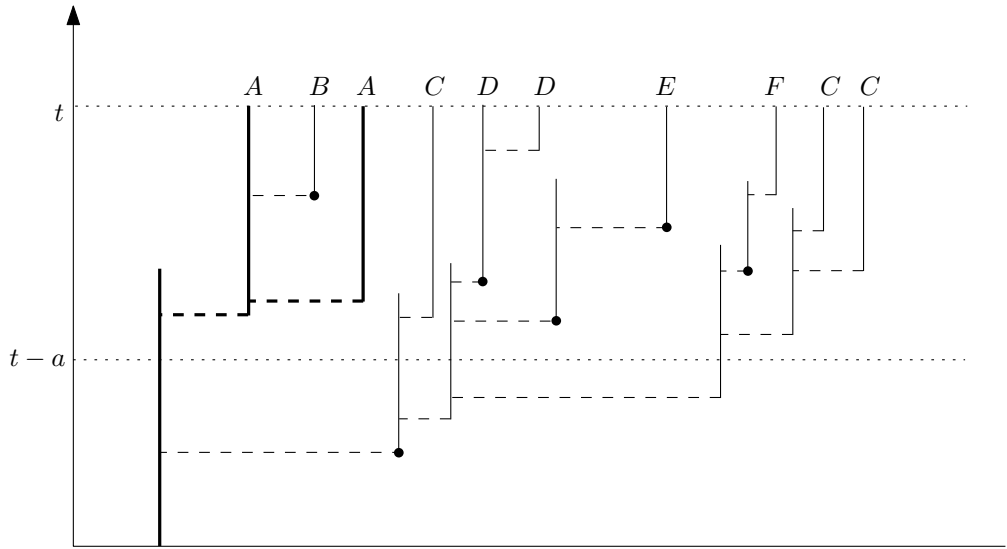}
\caption{An example of a splitting tree with mutations up to time $t$. The vertical axis represents time and the horizontal axis shows filiation (lengths of dashed lined are meaningless). Full circles represent mutations occurring at birth and thick lines, the clonal splitting tree of the ancestor up to time $t$. The different letters are the alleles of live individuals at time $t$.}
\label{figure_spectre des frequences}
\end{center}
\end{figure}

Without mutation, if $\Xi(t)$ is the number of extant individuals at time $t$, then the process $(\Xi(t),t\geq0)$ is a Crump-Mode-Jagers (CMJ) process or general branching process (see \cite{Jagers_BP_with_bio,Jagers1984a} and references therein). We use the formalism commonly employed in CMJ-processes. Denote by $I=\bigcup_{n=0}^{\infty}(\Nat^*)^n$ the Ulam-Harris family history space
with the convention $\Nat^0=\{0\}$.
The interpretation is that an individual $\mathbf{i}=(i_1,\dots,i_{n-1},i_n)\in I$ is the $i_n$-th child of the $i_{n-1}$-th child of ... of the $i_1$-th child of the ancestor, who is labeled $\{0\}$. For more details, the interested reader can see \cite{Jagers_BP_with_bio}. With each individual $\mathbf{i}\in I$, one associates a non-negative random variable $\lambda_\mathbf{i}$ (its life length), and a point process $\xi_\mathbf{i}$ called birth point process. The sequence $(\lambda_\mathbf{i},\xi_\mathbf{i})_{\mathbf{i}\in I}$ is assumed to be i.i.d. but $\lambda_\mathbf{i}$ and $\xi_\mathbf{i}$ are not necessarily independent.

In our particular case, the common distribution of lifespans is $\Lambda(\cdot)/b$ and conditional on its lifespan, the birth point process of an individual is distributed as a Poisson point process during its life. If we denote by $\xi$ the birth point process of the ancestor, its characteristic measure is then given by
\begin{equation}\label{defi_mu}
\mu(\dif t):=\Esp[\xi(\dif t)]=\dif t\Lambda\big((t,+\infty]\big).
\end{equation}
In the same way, if $\xi_m$ (resp $\xi_c$) is the birth point process of the mutant (resp. clonal) children of the ancestor, then by the thinning property of Poisson point processes, $\xi_m$ and $\xi_c$ are independent, $\xi=\xi_m+\xi_c$ and with obvious notation, $\mu_m(\dif t)=p\mu(\dif t)$ and $\mu_c(\dif t)=(1-p)\mu(\dif t)$.

\subsection{Basic properties about the process \texorpdfstring{$\Xi$}{Xi}}\label{rappels}
We consider here the model without mutation and recall some known facts about the CMJ-processes $\Xi$, which will be useful in the following.

We say that the process $\Xi$ is subcritical, critical or supercritical if $$m:=\int_{(0,\infty]}r\Lambda(\dif r)$$ is respectively less than, equal to or greater than $1$.
We set
\begin{equation}\label{def_psi}
\psi(\lambda):=\lambda-\int_{(0,\infty]}\left(1-e^{-\lambda
r}\right)\Lambda(\dif r).
\end{equation}
 Since this function is convex and satisfies $\psi(0^+)=0$, we can define $\eta$ as its largest root. Moreover, since $\psi'(0^+)=1-m$, when the process $\Xi$ is subcritical or critical, we have $\eta=0$ and when it is supercritical, $\eta$ is positive. In the last case, $\eta$ is called the \emph{Malthusian parameter} of the population as explained in the forthcoming Proposition \ref{convergence_Amaury}.
 
To obtain properties about the splitting trees, Lambert defined in \cite{Amaury_contour_splitting_trees} a contour process which characterized them. This contour process is a spectrally positive L\'evy process with Laplace exponent $\psi$.  
Let $W$ be the \emph{scale function} associated with it (see \cite[Ch. VII]{Levy_processes}), defined as the unique increasing continuous function $(0,\infty)\rightarrow(0,\infty)$ satisfying
\begin{equation}\label{defi_W}
  \intpos W(x)e^{-\lambda x}\dif x=\frac{1}{\psi(\lambda)}, \quad \lambda>\eta.
\end{equation}
In the entire paper, by convenience, we assume that the following hypothesis holds.
\begin{equation}\label{noatom}
\textrm{The measure }\Lambda\textrm{ has no atom.}\tag{H}
\end{equation}
According to \cite[p. 234]{Kyprianou2006}, under the hypothesis \eqref{noatom}, the function $W$ is continuously differentiable and thanks to the Lemma 4.1 in \cite{Splitting_trees_Immig}, we have
\begin{equation}\label{formule_derivee_W}
 W'(t)=bW(t)-\int_0^tW(t-x)\Lambda(\dif x),\quad t\geq0.
\end{equation}
%where $\star$ denotes the standard convolution product.
%
In fact, if \eqref{noatom} is not satisfied, most of the results stated in this paper still hold by replacing $W'(t)$ by $bW(t)-\int_0^tW(t-x)\Lambda(\dif x)$. We mainly choose to assume \eqref{noatom} in order to simplify the results and their proofs. 

Thanks to the scale function $W$, we deduce the one-dimensional marginals of $\Xi$ (see \cite[p. 293]{Splitting_trees_Immig})
\begin{equation}\label{loi_Xt}
  \p(\Xi(t)=0)=1-\frac{W'(t)}{bW(t)}
\end{equation}
and for $n\in\Nat^*$,
\begin{equation}\label{loi_Xt_2}
  \p(\Xi(t)=n)=\left(1-\frac{1}{W(t)}\right)^{n-1}\frac{W'(t)}{bW(t)^2}.
\end{equation}
In other words, conditional on being positive, $\Xi(t)$ is geometric with success probability $1/W(t)$.
In particular, for $t\geq0$,
\begin{equation}\label{esp_Xt}
\EE{\Xi(t)}=\frac{W'(t)}{b}.
\end{equation}
%
% and if $T$ denotes the time of extinction of the tree
% $$T:=\inf\{t\geq0;\Xi(t)=0\},$$
% we have for $t\geq0$
% \begin{equation}\label{loi_T}
%   \p(T>t)=\p(\Xi(t)\neq0)=\frac{W'(t)}{bW(t)}.
% \end{equation}
%
Let $\ext:=\left\{\Xi(t)\underset{t\rightarrow\infty}{\longrightarrow} 0\right\}$ be the extinction event of the splitting tree. We denote by $\peto:=\p(\cdot|\ext^c)$ the law of the process $\Xi$ conditional on its survival event.

\begin{prop}[Lambert \cite{Amaury_contour_splitting_trees}]
  \label{convergence_Amaury}
  We have $$\p(\ext)=1-\eta/b.$$
  Moreover, if $\Xi$ is supercritical ($m>1$), under $\peto$,
  \begin{equation}\label{epone3}e^{-\eta t}\Xi(t)\underset{t\rightarrow\infty}\longrightarrow \mathcal E\as
  \end{equation}
  where $\mathcal E$ is exponential with parameter $\pse$.
  \end{prop}
In fact, Lambert proved the convergence in distribution in \cite{Amaury_contour_splitting_trees} and a.s. convergence holds according to \cite{Nerman_supercritical_CMJ} (see \cite[p.285]{Splitting_trees_Immig}). The convergence \eqref{epone3} justifies why we call $\eta$ the Malthusian parameter of the population since $\Xi(t)$ grows like $e^{\eta t}$ on the non-extinction event.\\

Most of the results stated in Sections \ref{section_spectre_frequence} and \ref{nombre_de_mutations_subies} deal with long-time behaviors of several processes. To obtain them, we need the asymptotic behaviors of the scale function $W$ and of its derivative $W'$.
Different regimes appear depending on whether $\Xi$ is subcritical, critical or supercritical.
We record them in the following lemma. %When $\Xi$ is subcritical, we assume that $\Lambda$ has exponential moments and in the critical case that the its second moment is finite.

\begin{lem}
\begin{enumerate}[{\normalfont (i)}]
\item If $m>1$ (supercritical case), as $t\to\infty$, we have
$$W(t)\sim \frac{1}{\pse}e^{\eta t}\textrm{ and  }W'(t)\sim \frac{\eta}{\pse}e^{\eta t}.$$
\item If $m=1$ (critical case) and $\sigma^2:=\intpos r^2\Lambda(\dif r)<\infty$,
we have $W(t)\underset{t\to\infty}\sim \frac{2}{\sigma^2}t.$
If we also suppose that 
\begin{equation}\label{cond_W_critique}
\lim_{t\to\infty}t^2\int_t^\infty(x-t)\Lambda(\dif x)=0,
 \end{equation}
we have $\displaystyle\lim_{t\to\infty}W'(t)=\frac{2}{\sigma^2}.$ 
\item If $m<1$ (subcritical case),  $\displaystyle\lim_{t\to\infty}W(t)=\frac{1}{1-m}.$
If we also assume that there is a negative $\tilde\eta$ satisfying
\begin{equation}\label{condJagers3}
\psi(\tilde\eta)=0\textrm{ and }\intpos re^{-\tilde\eta r}\Lambda(\dif r)<\infty,
\end{equation}
then
$\displaystyle W'(t)\underset{t\to\infty}\sim\frac{\tilde\eta}{\psi'(\tilde\eta)}e^{\tilde\eta t}.$
\end{enumerate}
\label{W_et_W'}
\end{lem}
In the critical case, the condition \eqref{cond_W_critique} holds if $\Lambda$ has a finite third moment. Concerning the subcritical case, it is possible to find a negative root of $\psi$ because this function is convex, $\psi(0)=0$ and $\psi'(0)=1-m>0$. However, to have this root requires  
that the exponential moment $\intpos e^{\lambda r}\Lambda(\dif r)$ is finite for $\lambda$ large enough. %A simple example where condition \eqref{condJagers3} holds is the case of a Markovian birth and death process with birth rate $b$ and death rate $d$.
The proof of Lemma \ref{W_et_W'} is postponed to the appendix. 
\subsection{The clonal process}
In the sequel, an important role will be played by the clonal process $(\Xi_c(t),t\geq0)$ (c for clonal) where $\Xi_c(t)$ is the number of extant individuals at time $t$, carrying the same allele as the ancestor (see Figure \ref{figure_spectre des frequences}). Since mutations occur independently of $(\Xi(t),t\geq0)$, by the thinning property of Poisson processes, the process $\Xi_c$ defines a splitting tree with lifespan measure $$\Lambda_c:=(1-p)\Lambda.$$ Moreover, conditional on its arrival time in the population, each allele evolves like $\Xi_c$ and independently of other alleles.
% 
% If we compare the two processes $\Xi$ and $\Xi_c$, mutations only affect the births by decreasing the birth rate. On the contrary, in \cite{Champagnat2010,Champ_Lamb_2}, the mutations occur at a Poissonian rate during the life of individuals and in their clonal process, the mutations can be seen as an additional death rate.

According to the previous section, the process $\Xi_c$ is subcritical, critical or supercritical if $(1-p)m$ is respectively less than, equal to, or greater than $1$. In particular, when $\Xi$ is critical or subcritical, $\Xi_c$ is obviously subcritical.

 Furthermore, we associate with $\Xi_c$ the function $\psi_c$, which satisfies
\begin{equation}\label{def_psi_c}
\psi_c(\lambda):=\lambda-(1-p)\int_{(0,\infty]}\left(1-e^{-\lambda
r}\right)\Lambda(\dif r)=p\lambda+(1-p)\psi(\lambda).
\end{equation} 
Let $\eta_c$ be the largest root of $\psi_c$. When $\Xi_c$ is subcritical or critical, $\eta_c=0$ while in the supercritical case, $\eta_c>0$. Moreover, in the latter, by the definition of $\eta_c$ and by using \eqref{def_psi_c}, we have
$$\psi(\eta_c)=\frac{p\eta_c}{p-1}<0.$$ It implies that $\eta_c<\eta$ since $\psi$ is convex and its largest root is $\eta$.\\

Finally, all the properties about $\Xi$, stated in the paragraph \ref{rappels}, also hold for $\Xi_c$. To obtain them, it suffices to respectively replace $b$, $\psi$, $\eta$ and $W$ by $b(1-p)$, $\psi_c$, $\eta_c$ and $W_c$ where $W_c$ is the scale function associated with $\psi_c$, solution of
\begin{equation*}\label{defi_W_c}
\intpos W_c(x)e^{-\lambda x}\dif x=\frac{1}{\psi_c(\lambda)},\quad \lambda>\eta_c.
\end{equation*}
% \begin{rem}
%  One can find a closed-form formula between $W$ and $W_c$. Indeed, since $$\frac{\lambda}{\psi(\lambda)}-1=\frac{1}{1-p}\left(\frac{\lambda}{\psi(\lambda)}-1\right)\left(p\left(\frac{\lambda}{\psi(\lambda)}-1\right)+1\right),$$
%  $$W'=W'_c+pW'\star W'_c.$$
% It implies that
%  $$W=W_c+pW'\star W_c$$
%  since $W(0)=W_c(0)=1$.
%  %
% However, this formula is not useful.
%  \end{rem}

\section{Frequency spectrum}\label{section_spectre_frequence}
At a given time, for any allele, we call \emph{family} the set of all individuals that share this allele.
 To study the allelic partition of the population, we study its associated frequency spectrum, which, roughly speaking, sort the different families according to their sizes and their ages.
 More precisely, for $i\in \Nat^*$ and $a>0$, let $M_t^{i,a}$ be the number of alleles whose ages are less than $a$ and carried by $i$ individuals at $t$. Then, for fixed $0\leq a<t$, the sequence $(M_t^{i,a},i\geq1)$ is the frequency spectrum at time $t$ of families with ages less than $a$.
  
  Notice that $M_t^{i,t}$ is simply the number of alleles carried by $i$ particles at time $t$ (regardless of their ages) and that $\displaystyle M_t:=\sum_{i\geq1}M_t^{i,t}$ is the number of different alleles at time $t$.

In the example on Figure \ref{figure_spectre des frequences}, taking no account of ages of alleles, the frequency spectrum $\left(M_t^{i,t},i\geq1\right)$ is $(3,2,1,0,\dots)$ because three alleles ($B,E,F$) are carried by one individual, $A$ and $D$ are carried by two individuals and $C$ is the only allele carried by three individuals. If we are only interested in alleles younger than $a$ at time $t$, we have $\left(M_t^{i,a},i\geq1\right)=(3,1,0,\dots)$ since alleles $A$ and $D$ appeared in the population before $t-a$.\\

Although it is not possible to obtain the joint distribution of $(M_t^{i,a},i\geq1)$ for fixed $t$ and $a$ as in the 'Ewens sampling formula', we are able to get some properties of this frequency spectrum.

\subsection{Expected frequency spectrum}\label{moyenne spectre allelique}

We first give an exact expression of the expected frequency spectrum at any time $t$.
For $0<a<t$ and $i\geq1$, we denote by $M_t^{i,\dif a}$ the number of alleles carried by $i$ individuals at time $t$ and with ages in $[a-\dif a,a]$. The following proposition yields its expected value.

\begin{prop}\label{wxcv}
For $0<a<t$ and $i\geq1$, we have
\begin{equation}\label{expect_spectrum}\EE{M_t^{i,\dif a}}=\frac{p}{b(1-p)}W'(t-a)\left(1-\frac{1}{W_c(a)}\right)^{i-1}\frac{W_c'(a)}{W_c^2(a)}\ \dif a.\end{equation}
\end{prop}

\begin{proof}
  Conditional on $\Xi(t-a)$, thanks to the branching property and classical properties of Poisson point processes, $M_t^{i,\dif a}$ is the sum of $\Xi(t-a)$ i.i.d. random variables distributed as the number of atoms in the interval $[t-a,t-a+\dif a]$ of a Poisson point process with parameter $bp\p(\Xi_c(a)=i)$. Hence,
  $$\EE{M_t^{i,\dif a}}=\Esp[\Xi(t-a)]\ bp\ \p(\Xi_c(a)=i)\dif a.$$
  According to \eqref{loi_Xt_2} and \eqref{esp_Xt}, $\Esp[\Xi(t-a)]=W'(t-a)/b$ and $$\p(\Xi_c(a)=i)=\left(1-\frac{1}{W_c(a)}\right)^{i-1}\frac{W_c'(a)}{b(1-p)W_c(a)^2}$$
  and we obtain the desired result.
  \end{proof}

As a consequence of this proposition, we deduce the expected value of the frequency spectrum. In particular, when $\Xi$ is supercritical, we see that $M_t^{i,a}$ grows as $e^{\eta t}$ when $t\to\infty$, that is, with the same growth rate as $\Xi(t)$.
\begin{cor}\label{linkinpark}
For $a\leq t$ and $i\geq1$,
{\setlength\arraycolsep{2pt}
\begin{eqnarray}
\EE{M_t^{i,a}}&=&\frac{p}{b(1-p)}\int_0^a W'(t-x)\left(1-\frac{1}{W_c(x)}\right)^{i-1}\frac{W_c'(x)}{W_c^2(x)}\dif x\nonumber\\
&&\qquad\qquad\qquad\qquad+\frac{1}{b(1-p)}\left(1-\frac{1}{W_c(t)}\right)^{i-1}\frac{W_c'(t)}{W_c^2(t)}\1{t=a}\label{jwt}.
\end{eqnarray}}
Moreover, if $m>1$, for $a\geq0$, as $t\rightarrow\infty$,
\begin{equation}\label{oasis}
e^{-\eta t}\EE{M_t^{i,a}}\longrightarrow\frac{\eta}{b}\frac{p}{1-p}\frac{J^{i,a}}{\pse}
\end{equation}
where
$$J^{i,a}:=\int_0^a e^{-\eta u}\left(1-\frac{1}{W_c(u)}\right)^{i-1}\frac{W_c'(u)}{W_c^2(u)}\dif u.$$
\end{cor}
%The first assertion of this corollary still holds when $m\leq1$, that is, when $\Xi$ is subcritical or critical. 
In \cite{Pakes1989}, Pakes obtained a similar result: in Lemma 3.1.2, he computed the expected frequency spectrum for a Markov branching process where children can independently be mutants of their mothers with probability $p$.
In \cite[Chap 3]{TheseMathieu}, the expected frequency spectrum formula \eqref{jwt} is used to obtain the orders of magnitude of the sizes (resp. of the ages) of the largest families (resp. of the oldest families) at time $t$ when $t\to\infty$ (see also \cite{Champ_Lamb_Rich}).

\begin{proof}[Proof of Corollary \ref{linkinpark}]We obtain \eqref{jwt} by integrating \eqref{expect_spectrum} on $(0,a)$. The second term of the r.h.s. corresponds to $\p(\Xi_c(t)=i)$, that is, the probability that the progenitor has $i$ alive clonal descendants at time $t$.

  We get the convergence result \eqref{oasis} by using the dominated convergence theorem. 
   First, 
 %  integrating \eqref{defi_W} by parts, we have
%   $$\intpos e^{-\lambda x}W'(x)\dif x=\frac{\lambda}{\psi(\lambda)}-1$$ and using a Tauberian theorem as in \textbf{Lemma in \cite{Splitting_trees_Immig}}, 
as $t\rightarrow\infty$, using the Lemma \ref{W_et_W'}, $W'(t-x)$ is equivalent to $\frac{\eta}{\pse}e^{\eta(t-x)}.$
  Then, we obtain the convergence of the integral because $\left(1-\frac{1}{W_c(x)}\right)^{i-1}\frac{W_c'(x)}{W_c^2(x)}\leq b(1-p)$ and because there exists $C>0$ such that $e^{-\eta t}W'(t-x)\leq C$ for $t$ large enough.
  \end{proof}

\subsection{Convergence results}
In this paragraph, we suppose that the process $\Xi$ is supercritical ($m>1$) and are interested in improvements of the convergence result \eqref{oasis}. We recall that $\peto$ is the law of the process $\Xi$ conditional on its survival event.

The following proposition yields the asymptotic behavior, as $t\rightarrow\infty$ and under $\peto$, of the frequency spectrum $(M_t^{i,a},i\geq1)$. We also obtain the convergence of the relative abundances $M_t^{i,a}/M_t$ of families of sizes $i$ and ages less than $a$. 
\begin{prop}
  \label{nombre_mutants}
Under $\peto$, with probability 1,
\begin{equation}
\label{dutronc1}e^{-\eta t}M_t\underset{t\rightarrow\infty}\longrightarrow\frac{p }{1-p}J \mathcal E,
\end{equation}
\begin{equation}
\label{dutronc2}e^{-\eta t}M_t^{i,a} \underset{t\rightarrow\infty}\longrightarrow\frac{p}{1-p}J^{i,a}\mathcal E,
\end{equation}
and
\begin{equation}
\label{dutronc3}\frac{M_t^{i,a}}{M_t} \underset{t\rightarrow\infty}\longrightarrow\frac{J^{i,a}}{J}
\end{equation}
where
$$J:=\intpos e^{-\eta u}\frac{W'_c(u)}{W_c(u)}\dif u$$
and where $\mathcal E$ is the exponential random variable with parameter $\pse$ defined by \eqref{epone3}.
\end{prop}
Notice that \eqref{dutronc2} is consistent with \eqref{oasis} since $\p(\ext^c)=\eta/b$ and $\EE{\mathcal E}=1/\pse$. Moreover, we point out that the relative abundances $M_t^{i,a}/M_t$ converges to a deterministic limit.

Similar convergence results were obtained by Z. Ta\"{\i}b in \cite{Taib1992} who considered a similar model with more general birth point processes and mutations. However, his convergence theorems give non-explicit limits. In our particular case, we can find the distribution of the limits because the one-dimensional marginals and the asymptotic behavior of $\Xi$ are known (see above-mentioned Proposition \ref{convergence_Amaury} and formulas (\ref{loi_Xt}) and (\ref{loi_Xt_2})).

\begin{proof}[Proof of Proposition \ref{nombre_mutants}]
To prove the three convergence results, we use general properties about CMJ-processes counted by random characteristics developed by Jagers and Nerman (see  \cite{Jagers_BP_with_bio,Jagers1984a,Jagers1984}). We mainly follow ideas of Ta\"ib
and all the properties we use are recorded in the appendix of \cite{Taib1992}.
We first check that the general assumptions (C.1-4) of this appendix hold in our case: recalling that $\mu$ is the measure defined by \eqref{defi_mu}, we must have
\begin{enumerate}[(a)]
\item $\mu(\R^+)>1$ (supercritical case).
\item There is a number $\lambda_0>0$ such that $\intpos e^{-\lambda_0 u}\mu(\dif u)=1$.
\item $\beta:=\intpos ue^{-\lambda_0 u}\mu(\dif u)<\infty$.
\item $\mu$ is not supported by any lattice $k\Z$ for $k>0$.
\end{enumerate}
First, by the definition of $\mu$, (a) is satisfied since we have supposed that $m>1$ and point (d) is straightforward.
Moreover, the largest root $\eta$ of $\psi$ satisfies (b) seeing that
 \begin{equation}\label{malthusian}\intpos e^{-\eta u}\mu(\dif u)=\intpos e^{-\eta u}\Lambda((u,+\infty])\dif u=\int_{(0,\infty]}\Lambda(\dif r)\frac{1-e^{-\eta r}}{\eta}=\frac{\eta-\psi(\eta)}{\eta}=1.
 \end{equation}
Finally, let us check that (c) holds. Since $\Lambda$ is a finite measure with mass $b$, we have
$$\beta=\intpos u e^{-\eta u}\Lambda((u,+\infty])\dif u\leq b\intpos ue^{-\eta u}\dif u<\infty.$$
% {\setlength\arraycolsep{2pt}
% \begin{eqnarray}
% \beta&=&\intpos u e^{-\eta u}\dif u\int_{(u,\infty)}\Lambda(\dif r)=\int_{(0,\infty]}\Lambda(\dif r)\int_0^r u e^{-\eta u}\dif u\nonumber\\
% %
% &=&\int_{(0,\infty]}\Lambda(\dif r)\left(\frac{-re^{-\eta r}}{\eta}+\frac{1-e^{-\eta r}}{\eta^2}\right)=\frac{\psi'(\eta)-1}{\eta}+\frac{\eta-\psi(\eta)}{\eta^2}=\frac{\psi'(\eta)}{\eta}<\infty.\nonumber
% \end{eqnarray}}

We are now able to prove \eqref{dutronc2}. We follow the proof of Theorem 3.3 in \cite{Taib1992} using a random characteristic. 
We denote by $\tau_1<\tau_2<\cdots$ the successive birth times of the children of the ancestor, by $\rho_1,\rho_2,\dots$ the independent sequence of i.i.d. Bernoulli random variables with parameter $1-p$ saying if a child is a clone or a mutant and by $\Xi_c^1,\Xi_c^2,\dots$ the i.i.d. clonal population processes of the mutant children. Then, for $u\geq a$, we set the characteristic
\begin{equation*}
\chi(u):=\sum_{j\geq1}(1-\rho_j)\1{u-a\leq\tau_j<u}\1{\Xi_c^j(u-\tau_j)=i},
\end{equation*}
that is, $\chi(u)$ is the number of mutant children of the ancestor born between $u-a$ and $u$ and whose alleles are carried by $i$ individuals at $u$.
For $u\leq a$, we set $\chi(u)=0$.

In the same way, for any individual $k$ and $u>a$, denote by $\chi_k(u)$ the number of mutant children of $k$, born between $u-a$ and $u$ units of time after the birth of $k$ and whose alleles are carried by $i$ individuals $u$ units of time after the birth of $k$.
Hence, $(M_t^{i,a},t\geq0)$ can be counted by the characteristic $\chi$, meaning that for $t\geq0$,
\begin{equation}\label{charac}
M_t^{i,a}=\sum_{k=1}^{y_t}\chi_k(t-\alpha_k)+\1{t\leq a}\1{\Xi_c(t)=i}\end{equation}
where $\alpha_k$ denotes the birth time of $k$ and $y_t$ is the number of individuals born before $t$.
The second term in the r.h.s. of (\ref{charac}) pertains to the ancestor which is not counted by the characteristic. However, this term vanishes when $t>a$ and so disappears at the limit $t\rightarrow\infty$.

According to Theorem 5 in the appendix of \cite{Taib1992}, if there exists $\eta'<\eta$ such that
\begin{equation}\label{condition1}
\intpos e^{-\eta't}\mu(\dif t)<\infty
\end{equation}
and such that
\begin{equation}
\label{condition2}
\Esp\left[\sup_{u\geq0}e^{-\eta'u}\chi(u)\right]<\infty,
\end{equation}
then, under $\peto$,
$$\lim_{t\rightarrow\infty}\frac{M_t^{i,a}}{y_t}=\intpos \eta e^{-\eta u}\EE{\chi(u)}\dif u\as$$
However, conditioning by the point process $\xi_m$ which is independent from the $\Xi_c^j$'s, we have
{\setlength\arraycolsep{2pt}
\begin{eqnarray}
\intpos \eta e^{-\eta u}\EE{\chi(u)}\dif u&=&\intpos \eta e^{-\eta u}\dif u\Esp\left[\sum_{j\geq1}(1-\rho_j)\1{u-a\leq\tau_j<u}\1{\Xi_c^j(u-\tau_j)=i}\right]\nonumber\\
&=&\intpos \eta e^{-\eta u} \dif u\Esp\left[\int_{u-a}^u\p(\Xi_c(u-t)=i)\xi_m(\dif t)\right]\nonumber\\
&=&\intpos \eta e^{-\eta u} \dif u\int_{u-a}^u\p(\Xi_c(u-t)=i)\mu_m(\dif t)\nonumber\\
%
%&=& \intpos \mu_m(\dif t)\int_t^{t+a}\eta e^{-\eta u}\p(\Xi_c(u-t)=i)\dif u\nonumber\\
%
&=&\intpos \mu_m(\dif t) e^{-\eta t} \int_0^a\eta e^{-\eta u}\p(\Xi_c(u)=i)\dif u\nonumber
\end{eqnarray}}
using Fubini-Tonelli's theorem and a change of variables. Since $\mu_m=p\mu$ and $\eta$ satisfies (\ref{malthusian}), $\intpos e^{-\eta t}\mu_m(\dif t)=p$. Then, thanks to (\ref{loi_Xt_2}),
\begin{equation}\label{epone1}
\lim_{t\rightarrow\infty}\frac{M_t^{i,a}}{y_t}=\frac{p\eta}{b(1-p)}\int_0^a e^{-\eta u}\left(1-\frac{1}{W_c(u)}\right)^{i-1}\frac{W_c'(u)}{W_c^2(u)}\dif u\as
\end{equation}

We still have to find $\eta'<\eta$ such that conditions (\ref{condition1}) and (\ref{condition2}) are satisfied.
First, as in (\ref{malthusian}), for $\eta'\in(0,\eta)$, $\intpos e^{-\eta' u}\mu(\dif u)=1-\psi(\eta')/\eta'<\infty$.
Second, the characteristic $\chi$ is stochastically dominated by a Poisson process with parameter $b$, say $(N_t,t\geq0)$. Then, to prove
(\ref{condition2}), it is sufficient to show that $\Esp\left[\sup_{t\geq0}e^{-\eta't}N_t\right]$ is finite.
At the end of the proof of Proposition 5.1 in \cite[p.1028]{Champagnat2010}, the authors show that for $\kappa$ large enough, the process $A:=\left(\left(e^{-\eta't}N_t+\kappa\right)^2,t\geq0\right)$ is a supermartingale, which implies that for $y\geq0$,
$$\p\left(\sup_{t\geq0}e^{-\eta't}N_t\geq y\right)=\p\left(\sup_{t\geq0} A_t\geq (y+\kappa)^2\right)\leq\frac{\Esp[A_0]}{(y+\kappa)^2}=\frac{\kappa}{(y+\kappa)^2},$$
where the inequality is due to the maximal inequality \cite[p. 58]{Revuz1999}.
Then, $$\Esp\left[\sup_{t\geq0}e^{-\eta't}N_t\right]=\intpos \dif y\p\left(\sup_{t\geq0}e^{-\eta't}N_t\geq y\right)<\infty.$$
Furthermore, using again \cite[Thm 5]{Taib1992} with the characteristic $\chi'(u)=\1{0\leq u\leq\lambda}$, on $\peto$
\begin{equation}\label{epone2}\frac{\Xi_t}{y_t}=\frac{\sum_{k=1}^{y_t}\chi'_k(t-\alpha_k)}{y_t}\underset{t\rightarrow\infty}\longrightarrow\intpos \eta e^{-\eta x}\EE{\chi'(x)}\dif x=\intpos\eta e^{-\eta x}\dif x\int_{(x,\infty)}\frac{\Lambda(\dif r)}{b}=\frac{\eta}{b}
\end{equation}
because $\psi(\eta)=0$. Notice that in that case, condition (\ref{condition2}) is easily satisfied because $\chi'(u)\leq1$ for $u\geq0$.
Finally, using together (\ref{epone3}), (\ref{epone1}) and (\ref{epone2}) we get the convergence result \eqref{dutronc2}.

We now prove \eqref{dutronc1}. According to \cite[Thm 3.3]{Taib1992}, on $\peto$,
 $$\frac{M_t}{y_t}\underset{t\rightarrow\infty}\longrightarrow\intpos e^{-\eta x}\mu_m(\dif x)\left(1-\eta\intpos e^{-\eta u}\p(\Xi_c(u)=0)\right)\dif u=p\eta\intpos e^{-\eta u}\frac{W'_c(u)}{b(1-p)W_c(u)}\dif u.$$
By using the last display, \eqref{epone3} and \eqref{epone2}, on $\peto$, we have
$$e^{-\eta t}M_t\linf{t}\mathcal{E}\frac{p}{1-p}\intpos e^{-\eta u}\frac{W'_c(u)}{W_c(u)}\dif u.$$
%
% We then get \eqref{dutronc1} via an integration by parts since $W_c(0)=1$ according to Lemma 4.1 in \cite{Splitting_trees_Immig}. %\cite[Lem 4.1]{Splitting_trees_Immig}
%
Finally, \eqref{dutronc3} is straightforward from \eqref{dutronc1} and \eqref{dutronc2}.
\end{proof}

The following result deals with the asymptotic behavior of $M_t^{i,t}$, the number of alleles carried by $i$ individuals at time $t$.
\begin{prop}\label{allblack}
  Under $\peto$, with probability 1 as $t\rightarrow\infty$,
  $$\frac{M_t^{i,t}}{M_t}\longrightarrow%\frac{J^{i,\infty}}{J}=
  \frac{1}{i}\frac{\intpos e^{-\eta u}\left(1-\frac{1}{W_c(u)}\right)^{i}\dif u}{\intpos e^{-\eta u}\ln(W_c(u))\dif u}.$$
\end{prop}
Notice that the limit can be considered as a weighted Fisher log-series distribution, which is commonly used to describe species abundances \cite{Fisher1943}. More precisely, this limit can be written 
$$\frac{\frac{1}{i}\int_0^1(1-\omega)^i\dif \nu(\omega)}{\int_0^1\log(1/\omega) \dif\nu(\omega)}$$ where $\dif \nu(\omega)=e^{-\eta u(\omega)}\dif u(\omega)$ and $u(\omega)=W^{-1}(1/\omega)$.

\begin{proof}
Since $M_t^{i,t}$ can be counted by the random characteristic
$$\overline\chi(u):=\sum_{j\geq1}(1-\rho_j)\1{\tau_j<u}\1{\Xi_c^j(u-\tau_j)=i},$$
similar computations as in the proof of Proposition \ref{nombre_mutants} lead to the convergence  of $M_t^{i,t}/M_t$ to $J^{i,\infty}/J$ as $t\to\infty$.
We then get the result by integrating by parts $J$ and $J^{i,\infty}$ by noticing that $W_c(0)=1$ according to \cite[Lemma 4.1]{Splitting_trees_Immig}.

\end{proof}

\begin{rem}\label{jtmelo}
  In most cases, it is not possible to simplify the expressions of $J$ and $J^{i,a}$ appearing in the limits in Propositions \ref{nombre_mutants} and \ref{allblack} since the scale functions $W$ and $W_c$ are generally unknown.
  
  However, when the life lengths are exponentially distributed with parameter $d<b$ (in that case, $\Xi$ is simply a Markovian linear birth and death process with birth rate $b$ and death rate $d$), we know from \cite[p393]{Amaury_contour_splitting_trees} that 
  \begin{equation}\label{Wcexpo}
  W_c(x)=\left\{
  \begin{array}{cl}
  \frac{d-(1-p)be^{((1-p)b-d)x}}{d-(1-p)b} &\textrm{ if } d\neq(1-p)b\\
  1+dx &\textrm{ if }d=(1-p)b
  \end{array}\right..
    \end{equation}
%    and integrals are still uncomputable.
In that case, as it was recently done in \cite{Kimmel2013}, thanks to simple changes of variables, the integrals $J$ and $J^{i,\infty}$ can be expressed as Gauss hypergeometric functions or as confluent hypergeometric functions \cite{MR0167642} whose parameters only depend on $b$, $d$, $p$ and $i$. 
Then, $J$ and $J^{i,\infty}$ can be numerically computed. 
% More precisely, the Malthusian parameters associated with $\Xi$ and $\Xi_c$ are respectively $\eta=b-d$ and $\eta_c=b(1-p)-d$. We denote their ratio by $r:=\eta/\eta_c$.
% We then have $$J_i=\left\{
% \begin{array}{ll}
% \frac{\Gamma(r)i!}{\Gamma(r+i+1)\eta_c}F\left(i,r,i+1+r;\frac{d}{b(1-p)}\right)&\textrm{ if } b(1-p)>d\\
% %
% \frac{\Gamma(\eta/\eta_c)i!}{\Gamma(\eta/\eta_c+i+1)\eta_c}F\left(i,\frac{\eta}{\eta_c},i+1+\frac{\eta}{\eta_c};\frac{d}{b(1-p)}\right)&\textrm{ if } b(1-p)=d\\
% %
% -\left(\frac{b(1-p)}{d}\right)^i\frac{\Gamma(|r|)i!}{\Gamma(|r|+i+1)\eta_c}F\left(i,|r|,i+1+|r|;\frac{b(1-p)}{d}\right)&\textrm{ if } b(1-p)<d 
% \end{array}
% \right.$$

Finally, when $\Xi$ is a pure birth process, that is, when $\Lambda(\cdot)/b=\delta_\infty$, 
there are simple expressions of these integrals. Indeed, in that case, we have
$\psi(\lambda)=(\lambda-b)\1{\lambda>0}$, $\eta=b$, $W_c(x)=e^{(1-p)bx}$ from (\ref{Wcexpo}) and obviously, $\p(\ext)=0$ and $\peto=\p$.
Moreover, easy computations lead to simple expressions of $J$ and $J^{i,a}$ and we obtain convergence results such as
   $$\frac{M_t^{i,t}}{M_t}\linf{t} 
%\frac{1}{i(1-p)}\int_0^1 \left(1-y^{1-p}\right)^i\dif y
\frac{1}{i(1-p)}\sum_{k=0}^i(-1)^k\frac{1}{1+(1-p)k}\as$$    
\end{rem}

\section{Number of undergone mutations}\label{nombre_de_mutations_subies}
Let us say that the allele of the progenitor is of \emph{type 0}.
Then, recursively, for $k\geq1$, we say that an allele is of \emph{type $k$} if it is carried by a mutant child of an individual carrying an allele of type $k-1$. Equivalently, they are alleles that have been affected by $k$ more mutations than the ancestral allele. 
For $t\geq0$ and $i\geq0$, we denote by $L_i(t)$ the number of alleles of type $i$ at time $t$ and by $K_i(t)$ the number of individuals that carry such alleles at time $t$.
% In epidemiology, the alleles may model several strains of a disease/virus. Then, $L_i$ is the number of strains of a virus that have undergone $i$ more mutations than the original strain; $K_i$ is the number of individuals carrying such strains of the virus.

In this section, we obtain similar results as those of Section \ref{section_spectre_frequence}, that is, we compute the expected values $\EE{K_i(t)}$ and $\EE{L_i(t)}$ for fixed $i$ and $t$. We then study the asymptotic behavior of $K_i(t)$ as $t\to\infty$ when the clonal process is supercritical. 

The main argument we use to get these results is to consider $(K_i(t),i\geq0,t\geq0)$ as a multitype CMJ-process. More precisely, regardless of its type, every individual has a lifespan distributed as $\Lambda(\cdot)/b$ and an individual with type $i$ gives birth to individuals of the $i$-th type at rate $b(1-p)$ and to individuals of type $i+1$ at rate $bp$. Moreover, the process $(K_i,i\geq1)$ belongs to the class of \emph{reducible} multitype processes since an individual of type $i$ cannot give birth to an individual of type $j$ with $j<i$. Then, if at some time, there is no more individuals of type $i$, no such individual will reappear in the population. 

Even though the reducible branching processes are less studied than the irreducible processes (which enable the use of Perron-Frobenius theory of positive matrices),
they are useful in the model we consider. Indeed, if we are interested in the study of the process $K^i$ for a given $i$, since the types cannot decrease, we can transform a problem with an infinite and countable number of types to a problem with a finite number of possible types by studying the multitype process $(K_j,j=1,\dots, i)$.

The reader interested in multitype CMJ-processes can refer to the third and fourth chapters of Mode's book \cite{Mode1971}. 

\subsection{Expected number of mutations}
For $f,g$ two continuous functions on $[0,+\infty)$, we denote by $f\star g$ the standard convolution product
$$f\star g(t)=\int_0^tf(t-x)g(x)\dif x,\quad t\geq0.$$
We use the following notation for the consecutive convolutions of a function. For a continuous function $f$, let $f^{\star(1)}:=f$ and for $i\geq2$, $f^{\star(i)}:=f\star f^{\star(i-1)}$.
In the following result, we give simple expressions of the mean values of $K_i(t)$ and $L_i(t)$ that are entirely determined by the functions $W_c$ and $W'_c$.  

\begin{prop}\label{nombre_de_types}
 We fix $t>0$. We then have
 \begin{equation}\label{m3d}\EE{K_i(t)}=\frac{1}{b(1-p)}\left(\frac{p}{1-p}\right)^i(W_c')^{\star(i+1)}(t),\quad i\geq0
   \end{equation}
 \begin{equation}\label{m3d3}\EE{L_0(t)}=\frac{1}{b(1-p)}\frac{W'_c(t)}{W_c(t)}
 \end{equation}
 \begin{equation}\label{m3d2}\EE{L_i(t)}=\frac{1}{b(1-p)}\left(\frac{p}{1-p}\right)^i(W_c')^{\star(i)}\star\frac{W'_c}{W_c}(t),\quad i\geq1.
 \end{equation}
 \end{prop}
 
   In \cite{Mode1971}, C. Mode considers general multitype branching processes. In paragraph 4.4, he gives an equation satisfied by the Laplace transform of $\EE{K_i(\cdot)}$. In most cases, this equation cannot be solved but it enables to obtain the asymptotic behavior of $\EE{K_i(t)}$ as $t\to\infty$. In our particular case, one can solve the equation of Mode to obtain the Laplace transform of $\EE{K_i(\cdot)}$ and to compute $\EE{K_i(t)}$ thanks to an inverse Laplace transform.
   Nevertheless, in the following proof, we prefer to use a more direct approach which avoids long computations.

\begin{proof}[Proof of Proposition \ref{nombre_de_types}]
 We prove \eqref{m3d} by induction on $i\geq0$. For $i=0$, according to \eqref{esp_Xt}, 
 $\EE{K_0(t)}=\EE{\Xi_c(t)}=\frac{W_c'(t)}{b(1-p)}$.
 We now suppose that the result is true at rank $i-1$.
 Denote by $K_i^{\dif a}(t)$ the number of individuals of type $i$ that are alive at time $t$ and whose alleles appeared in the population between times $t-a$ and $t-a+\dif a$. 
 Then, in the same manner as in the proof of Proposition \ref{wxcv}, using the branching property, conditional on $K_{i-1}(t-a)=n$, $K_i^{\dif a}(t)$ is the sum of $n$ independent random variables. These random variables are distributed as the number at time $a$, of all the clonal descendants of mutant individuals born in a time interval $[0,\dif a]$.
 Then,
 $$\EE{K_i^{\dif a}(t)}=\EE{K_{i-1}(t-a)}\EE{\Xi_c(a)}bp\dif a.$$
 Finally, using the induction hypothesis and integrating over $a\in(0,t)$, we have
 $$\EE{K_i^{a}(t)}=\frac{1}{b(1-p)}\left(\frac{p}{1-p}\right)^{i-1}\int_0^t\left(W_c'\right)^{\star(i)}(t-a)\frac{W'_c(a)}{b(1-p)}bp\dif a,$$
 which is what we wanted to prove.
 
 Since $L_0(t)=\1{\Xi_c(t)>0}$, \eqref{m3d3} is straightforward from \eqref{loi_Xt}.
 We prove \eqref{m3d2} with similar techniques as we proved \eqref{m3d}. The mean number of alleles of type $i$ with age $a$ at time $t$ is
 $$\EE{L^{\dif a}_i(t)}=\EE{K_{i-1}(t-a)}bp\p(\xi_c(a)>0)\dif a.$$
 We then get the result thanks to \eqref{loi_Xt}, \eqref{m3d} and by integrating the last display on $[0,t]$.
\end{proof}

 From Proposition \ref{nombre_de_types}, we obtain the asymptotic behaviors of $\EE{K_i(t)}$ and $\EE{L_i(t)}$ as $t\to\infty$. Different regimes appear depending on the class of criticality of the clonal process $\Xi_c$.
 
 \begin{cor}\label{asympt_nombre_moyen_types}
 We suppose that one of the following hypotheses holds
 \begin{itemize}
 \item $(1-p)m>1$,
 \item $(1-p)m=1$, $\sigma^2<\infty$ and $\lim_{t\to\infty}t^2\int_t^\infty(x-t)\Lambda(\dif x)=0$,%\eqref{cond_W_critique} is satisfied,
 \item $(1-p)m<1$ and $\psi_c$ has a negative root $\tilde \eta_c$ satisfying $\intpos re^{-\tilde \eta_cr}\Lambda(\dif r)<\infty$.%\eqref{condJagers3}.
 \end{itemize}
  Then, for $i\geq0$, we have 
$$\EE{K_i(t)}\underset{t\to\infty}\sim \frac{1}{b(1-p)i!}\left(\frac{p}{1-p}\right)^iC_p^{i+1}\cdot t^ie^{\eta_p t}$$
 where 
 $$\eta_p:=\left\{
 \begin{array}{cc}
\eta_c&\textrm{ if } (1-p)m>1\\                    
0&\textrm{ if } (1-p)m=1\\
\tilde\eta_c&\textrm{ if } (1-p)m<1
\end{array}\right.
%\quad
%
\textrm{ and }
C_p:=\left\{
 \begin{array}{cc}
\displaystyle\frac{\eta_c}{\psi'_c(\eta_c)}&\textrm{ if } (1-p)m>1\\              
\displaystyle\frac{2}{(1-p)\sigma^2}&\textrm{ if } (1-p)m=1\\
\displaystyle\frac{\tilde\eta_c}{\psi'_c(\tilde\eta_c)}&\textrm{ if } (1-p)m<1
\end{array}\right..
$$
\end{cor}

 \begin{proof}%[Proof of Corollary \ref{asympt_nombre_moyen_types}]
 This result is a consequence of the following lemma thanks to Proposition \ref{nombre_de_types} and Lemma \ref{W_et_W'} which gives the asymptotic behaviors of $W'_c(t)$ as $t\to\infty$ in all cases. %$e^{-\eta_ct}W'_c(t)\to\eta_c/\psi_c'(\eta_c)$ as $t\to\infty$. 
\end{proof}
 
\begin{lem}\label{convo1}
 Let $f$ be a non-negative continuous function on $[0,+\infty)$ such that for some $a\in\R$ and $l\geq0$,
 \begin{equation}
 \label{train2}e^{at}f(t)\linf{t}l.
  \end{equation}
 Then, for $i\geq1$
 \begin{equation*}\label{train}
 \frac{e^{at}}{t^{i-1}}f^{\star(i)}(t)\linf{t}\frac{l^i}{(i-1)!}.
   \end{equation*}
\end{lem}

\begin{proof}
 We use induction on $i\geq1$ to prove the two properties
 \begin{equation}\label{train3}
 S_i:=\sup_{t\in(0,+\infty)}\frac{e^{at}}{t^{i-1}}f^{\star(i)}(t)<\infty\quad\textrm{ and }\quad\frac{e^{at}}{t^{i-1}}f^{\star(i)}(t)\linf{t}\frac{l^i}{(i-1)!}.
   \end{equation}
 This is obvious when $i=1$ since $f$ is continuous and since we assume \eqref{train2}.
 We now suppose that \eqref{train3} holds for an integer $i\geq1$.
 By a change of variables, we have
 $$f^{\star(i+1)}(t)=\int_0^tf(t-x)f^{\star(i)}(x)\dif x=t\int_0^1f(t(1-x))f^{\star(i)}(tx)\dif x.$$
Hence,
\begin{equation}\label{xmen}t^{-i}e^{at}f^{\star(i+1)}(t)=\int_0^1\left(e^{at(1-x)}f(t(1-x))\right)\left(\frac{e^{atx}}{(tx)^{i-1}}f^{\star(i)}(tx)\right)x^{i-1}\dif x.
 \end{equation}
Thus, we first have
$S_{i+1}\leq S_1S_i\int_0^1x^{i-1}\dif x<\infty$. Moreover, as $t\to\infty$, the r.h.s. of \eqref{xmen} converges to $\displaystyle l\frac{l^{i}}{(i-1)!}\int_0^1x^{i-1}\dif x=\frac{l^{i+1}}{i!}$ by using \eqref{train2}, the recurrence hypothesis and the dominated convergence theorem (the integrand is upper-bounded by $S_1S_ix^{i-1}$ which is integrable on $(0,1)$).
\end{proof}

We next determine the asymptotic behavior of $\EE{L_i(t)}$ as $t\to\infty$.
 \begin{cor}\label{asympt_nombre_moyen_types_2}
\begin{enumerate}[{\normalfont (i)}]
Let $i$ be a positive integer.
 \item If $(1-p)m>1$, we have
 $$\EE{L_i(t)}\underset{t\to\infty}\sim \frac{J_c}{b(1-p)(i-1)!}\left(\frac{p}{1-p}\right)^i\left(\frac{\eta_c}{\psi'_c(\eta_c)}\right)^{i}\cdot t^{i-1}e^{\eta_ct}$$
 where $J_c:=\intpos e^{-\eta_c u}\frac{W'_c(u)}{W_c(u)}\dif u.$
 \item If $(1-p)m=1$, if $\sigma^2<\infty$ and if \eqref{cond_W_critique} holds,
 $$\EE{L_i(t)}\underset{t\to\infty}\sim C_i\ {t^{i}}\ln t$$
 where $C_i=\frac{1}{b(1-p)(i-1)!}\left(\frac{2p}{\sigma^2(1-p)^2}\right)^i$.
 \item If $(1-p)m<1$ and if there is $\tilde \eta_c$ satisfying \eqref{condJagers3},
$$\EE{L_i(t)}\underset{t\to\infty}\sim
\frac{1-(1-p)m}{b(1-p)i!}\left(\frac{p}{1-p}\right)^i\left(\frac{\tilde \eta_c}{\psi'(\tilde \eta_c)}\right)^{i+1}\cdot t^{i}e^{\tilde \eta_ct}.$$ 
  \end{enumerate}
   \end{cor}
\begin{proof}
We begin by proving the subcritical case. According to Lemma \ref{W_et_W'}(iii), $W'_c(t)/W_c(t)$ behaves as
$(1-(1-p)m)\frac{\tilde\eta_c}{\psi'(\tilde \eta_c)}e^{\tilde \eta_c t}$ as $t\to\infty$. Then point (iii) stems from a slight modification of Lemma \ref{convo1}.

Concerning the critical case, according to Lemmas \ref{W_et_W'} and \ref{convo1}, 
for any $\eps>0$, there is $A>0$ such that, if $t-x\geq A$, then
$$\left|\frac{(W_c')^{\star(i)}(t-x)}{t^i}-\underbrace{\frac{1}{(i-1)!}\left(\frac{2}{(1-p)\sigma^2}\right)^i}_{:=A_i}\right|\leq \eps.$$
Then, for $\eps>0$ and $t\geq A$, we have
$$t^{-i}(W_c')^{\star(i)}\star\frac{W'_c}{W_c}(t)=I_1(t)+I_2(t)$$
where 
$$I_1(t):=\int_0^{t-A}\!\!t^{-i}(W_c')^{\star(i)}(t-x)\frac{W'_c(x)}{W_c(x)}\dif x\in\left[(A_i-\eps)\int_0^{t-A}\!\frac{W'_c(x)}{W_c(x)}\dif x,(A_i+\eps)\int_0^{t-A}\!\frac{W'_c(x)}{W_c(x)}\dif x\right]$$
and
$$I_2(t):=\int_0^{A}t^{-i}(W_c')^{\star(i)}(x)\frac{W'_c(t-x)}{W_c(t-x)}\dif x\leq t^{-i}b(1-p)\int_0^{A}(W_c')^{\star(i)}(x)\dif x$$
where we have used that $\frac{W'_c(x)}{W_c(x)}=b(1-p)\p(\Xi_c(t)>0)$. 
Hence, as $t\to\infty$, $I_2(t)$ vanishes and $I_1(t)$ is equivalent to $A_i\ln(t)$ since
we know from  Lemma \ref{W_et_W'}(ii) that $W'_c(x)/W_c(x)\sim 1/x$
as $x\to\infty$. 
Thus,
$$\EE{L^i(t)}=\frac{1}{b(1-p)}\left(\frac{p}{1-p}\right)^i(W_c')^{\star(i)}\star\frac{W'_c}{W_c}(t)\underset{t\to\infty}\sim\frac{A_i}{b(1-p)}\left(\frac{p}{1-p}\right)^i\ t^i\ln t.$$

Finally, we get the asymptotic behavior in the supercritical case by a direct application of the following Lemma \ref{convo2}. Indeed, according to the proof of Lemma \ref{asympt_nombre_moyen_types_2}, $(W_c')^{\star(i)}(t)$ grows as $t^ie^{\eta_ct}$ as $t\to\infty$ and from Lemma \ref{W_et_W'} (i), we have $\frac{W'_c(t)}{W_c(t)}\to\eta_c$ as $t\to\infty$.
\end{proof}
 
 \begin{lem}\label{convo2}
 Let $f,g$ be two continuous and positive functions such that $f(t)$ converges as $t\to\infty$ and such that for some $a>0$, 
 \begin{equation}\label{convo}
 \lim_{t\to\infty}\frac{e^{-at}}{t^i}g(t)=l'.
   \end{equation}
 Then, as $t\to\infty$, $$f\star g(t)\sim l'e^{at}t^i\intpos e^{-ax}f(x)\dif x.$$
 \end{lem}
 \begin{proof}
  We have $$\frac{e^{-at}}{t^i}f\star g(t)=\intpos e^{-ax}f(x)h(t,x)\dif x$$
  where $h(t,x):=\frac{e^{-a(t-x)}g(t-x)}{t^i}\1{0\leq x\leq t}$. Moreover, $h(t,x)\linf{t}l'\1{x\geq0}$ and according to \eqref{convo} $$h(t,x)\leq \frac{e^{-a(t-x)}g(t-x)}{(t-x)^i}\1{x\leq t}\leq M$$ where $M$ is a constant that does not depend on $t$ and $x$. Finally, since $f(t)$ converges as $t\to\infty$, $\intpos f(x)e^{-ax}\dif x$ is finite and the result stems from the dominated convergence theorem. 
  \end{proof}

  \subsection{Asymptotic behavior of \texorpdfstring{$K_i(t)$}{Ki(t)} when \texorpdfstring{$\Xi_c$}{Xic} is supercritical}
  In Corollary \ref{asympt_nombre_moyen_types}, we proved that $\EE{K_i(t)}$ grows as $t^ie^{\eta_ct}$ as $t\to\infty$ when the clonal process is supercritical.
  We now improve that convergence result by studying the asymptotic behavior of $K_i(t)$.
  Notice that in the subcritical and critical cases, the process $K_i$ goes extincted a.s.
  
 \begin{thm}\label{CV_types}
 We suppose that $(1-p)m>1$ and that $\sigma^2:=\intpos\Lambda(\dif z)z^2<\infty$.
Then, for $i\geq0$, almost surely and in quadratic mean, we have 
  $$\frac{e^{-\eta_ct}}{t^{i}}K_i(t)\underset{t\to\infty}{\longrightarrow} \kappa_i$$
  where 
  $\p(\kappa_i=0)=1-\frac{\eta_c}{b(1-p)}$ and conditional on being non zero, $\kappa_i$ is exponential with mean 
  $$\frac{1}{i!}\left(\frac{p}{1-p}\right)^i\frac{\eta_c^i}{\psi'_c(\eta_c)^{i+1}}.$$    
 \end{thm}

   \begin{proof}
  The case $i=0$ is straightforward since $K_0=\Xi_c$ and the result holds according to Proposition \ref{convergence_Amaury}. Then, from now on, we suppose that $i\geq1$.
  
 The proof is divided in three steps. In the first one, we prove that the a.s. and $L^2$ convergences hold. We then identify the law of the limit: we find conditions characterizing its Laplace transform (step 2) and then exhibit the random variable which satisfies these conditions (step 3). 
 \paragraph{Step 1: Proof of the convergence}\quad\\
  To prove the a.s. and quadratic mean convergences toward $\kappa_i$ as $t\to\infty$, we use Theorem 4.3 (iv) in \cite[p.173]{Mode1971}.
  Actually, in this book, C. Mode only proved that $t^{-1}e^{-\eta_c t}K_1(t)$ converges as $t\to\infty$ but the same result holds about $K_i$ for $i\geq2$ by using similar techniques.
   
   Let us check the hypotheses of Theorem 4.3 of \cite{Mode1971}. First, we have to show that $N$, the number of children of the ancestor, is square integrable. Conditional on having a life length $z$, $N$ is a Poisson random variable with parameter $bz$. Then,
   $$\EE{N^2}=b^{-1}\intpos\Lambda(\dif z)((bz)^2+bz)=b\sigma^2+m<\infty.$$
   Second, let us consider an individual of type $i$ whose life length is infinite. Denote by  
$\mu_{i}(t)$ her mean number of children that are born before she reaches age $t$. 
   We need to prove that $\max_{i}\mu_{i}(t)$ is zero when $t=0$ and finite when $t>0$. In our case, it is obvious since for any $i\geq0$, $\mu_{i}(t)=bt$. \\ 
   Finally, if $\zeta$ is the life length of an individual, we have to check that $\intpos\dif t\p(\zeta\geq t)^p$ is finite for $p=1$ and also for some $p>1$.
   We have
   $$\intpos\dif t\p(\zeta\geq t)^p\leq\intpos\dif t\p(\zeta\geq t)=\EE{\zeta}=\frac{m}{b}<\infty.$$

\paragraph{Step 2: Characterization of the Laplace transform of $\kappa_i$}\quad\\
For $i\geq1$, denote by $\varphi_i$ the Laplace transform of $\kappa_i$
  $$\varphi_i(a):=\EE{e^{-a\kappa_i}},\quad a>0.$$
% We first find conditions satisfied by $\varphi_i$. Thanks to the first step of the proof, we know that the convergence toward $\kappa_i$ holds in quadratic mean. Combining this with Corollary \ref{asympt_nombre_moyen_types}, it implies that
  $\EE{\kappa_i}=\frac{1}{b(1-p)}\frac{1}{i!}\left(\frac{p}{1-p}\right)^i\left(\frac{\eta_c}{\psi'_c(\eta_c)}\right)^{i+1}.$
In particular, $\kappa_i$ is almost surely finite and
\begin{equation}\label{indoc}
\varphi_i(0)=1 \textrm{ and } \varphi_i \textrm{ is differentiable at }0 \textrm{ with } \varphi_i'(0)=\frac{1}{b(1-p)}\frac{1}{i!}\left(\frac{p}{1-p}\right)^i\left(\frac{\eta_c}{\psi'_c(\eta_c)}\right)^{i+1}.
 \end{equation}
We now want to prove that for $i\geq1$ and $a>0$, we have
  \begin{equation}\label{equation_carac_kappa_i}
   \varphi_i(a)=\intpos\frac{\Lambda(\dif z)}{b}\exp\left\{b(1-p)\left(\int_0^z\varphi_i(ae^{-\eta_cu})\dif u-z\right)\right\}.
  \end{equation}
  To show that the latter holds, we look at the children of the ancestor. More precisely, conditional on having a lifespan $z$, before time $t$, the ancestor has a number of mutant (resp. clonal) children distributed as a Poisson random variable with parameter $bp(z\!\wedge\! t)$ (resp. $b(1-p)(z\!\wedge\! t)$). Moreover, these two random variables are independent. 
 
 For $z>0$, $t\geq0$ and $k,l\in\Nat$, by the branching property and classical properties about Poisson processes, conditional on the event ``during its lifespan $z$ and before age $t$, the ancestor has $k$ mutant and $l$ clonal children '', we have
  
  $$K_i(t)\overset{\textrm{(d)}}=\sum_{q=1}^k \widetilde K^{q}_i(t-U_q)+\sum_{r=1}^l \widehat K_i^r(t-V_r)$$
  where
  \begin{itemize}
   \item the processes $\widetilde K_i^1, \dots,\widetilde K_i^k$ are i.i.d. and distributed as $K_{i-1}$,
   \item the processes $\widehat K_i^1, \dots,\widehat K_i^k$ are i.i.d. and distributed as $K_{i}$,
   \item the random variables $U_1,\dots,U_p,V_1,\dots,V_q$ are i.i.d. and uniform in $[0,z\!\wedge\! t]$.
  \end{itemize}
Moreover, all these quantities are mutually independent. Hence, for $s\in[0,1]$, we have
{\setlength\arraycolsep{2pt}
\begin{eqnarray*}
\EE{s^{K_i(t)}}&=&\intpos\frac{\Lambda(\dif z)}{b}\sum_{k,l\geq0}e^{-bz\!\wedge\! t}\frac{(bp(z\!\wedge\! t))^k}{k!}\frac{(b(1-p)(z\!\wedge\! t))^l}{l!}\EE{s^{K_{i-1}(t-U_1)}}^k\EE{s^{K_{i}(t-V_1)}}^l\\
&=&\intpos\!\!\frac{\Lambda(\dif z)}{b}\exp\left\{-bz\!\wedge\! t+pb(z\!\wedge\! t)\EE{s^{K_{i-1}(t-U_1)}}+(1-p)b(z\!\wedge\! t)\EE{s^{K_{i}(t-V_1)}}\right\}.
\end{eqnarray*}}
Thus, if we denote by $G_{i,t}$ the probability generating function of $K_i(t)$, we have for $s\in[0,1]$
\begin{equation}\label{equa_generatrice}G_{i,t}(s)=\intpos\frac{\Lambda(\dif z)}{b}\exp\left\{-b(z\!\wedge\! t)+pb\int_0^{z\!\wedge\! t}G_{i-1,t-u}(s)\dif u+(1-p)b\int_0^{z\!\wedge\! t} G_{i,t-v}(s)\dif v\right\}.
 \end{equation}
According to the first step of the proof, as $t\to\infty$, $t^{-i}e^{-\eta_ct}K_i(t)$ converges a.s. and then in distribution to $\kappa_i$. Then, 
\begin{equation*}
G_{i,t}\left(e^{-at^{-i}e^{-\eta_ct} }\right)=\EE{e^{-aK_i(t)t^{-i}e^{-\eta_ct}}}\linf{t}\varphi_i(a),\nonumber\label{eq_G1}
\end{equation*}
\begin{equation}
G_{i-1,t-u}\left(e^{-at^{-i}e^{-\eta_ct} }\right)\linf{t}1,\label{eq_G2}
\end{equation}
and
\begin{equation}
G_{i,t-v}\left(e^{-at^{-i}e^{-\eta_ct} }\right)\linf{t}\varphi_i(ae^{-\eta_cv}).\label{eq_G3}
\end{equation}
We set
$$C(t,z):=-b(z\!\wedge\! t)+pb\int_0^{z\!\wedge\! t}G_{i-1,t-u}\left(e^{-at^{-i}e^{-\eta_ct}}\right)\dif u+(1-p)b
\int_0^{z\!\wedge\! t} G_{i,t-v}\left(e^{-at^{-i}e^{-\eta_ct}} \right)\dif v.$$
Then, since for all $t\geq0$ and $i\in\Nat$, $G_{i,t}$ and $G_{i-1,t}$ are less than $1$, using the dominated convergence theorem and equations \eqref{eq_G2} and \eqref{eq_G3}, we get
$$\lim_{t\to\infty}C(t,z)=-b(1-p)z+(1-p)b\int_0^z\varphi_i(ae^{-\eta_cv})\dif v$$  
Finally, if we apply \eqref{equa_generatrice} with $s=e^{-at^{-i}e^{-\eta_ct}}$ and let $t$ go to $\infty$, we obtain \eqref{equation_carac_kappa_i}. Indeed, we can again use the dominated convergence theorem since $C(t,z)\leq0$ for $t\geq0$ and $z>0$ and since $\Lambda$ is a finite measure.\\

In summary, we have proved that $\varphi_i$ satisfies \eqref{indoc} and \eqref{equation_carac_kappa_i}. In fact, according to the following Lemma \ref{unicite} whose proof is postponed to the end of the section, $\varphi_i$ is the only function that fulfills these two conditions, that is, $\kappa_i$ is characterized by \eqref{indoc} and \eqref{equation_carac_kappa_i}.
  \begin{lem}\label{unicite}
Suppose that $\Lambda$ satisfies the hypotheses of Theorem \ref{CV_types} and let $A\in(0,+\infty)$. Then, there is a unique function $f:[0,+\infty)\to[0,1]$ satisfying the following properties
 \begin{enumerate}[{\normalfont (i)}]
  \item $f$ is continuous on $[0,+\infty)$ and $f(0)=1$.
  \item $f$ is differentiable at $0$ and $f'(0)=A$.
  \item For any $a>0$, $f(a)=\intpos\frac{\Lambda(\dif z)}{b}\exp\left\{b(1-p)\int_0^z(f(ae^{-\eta_cu})-1)\dif u\right\}.$
 \end{enumerate}
 \end{lem}
 
\paragraph{Step 3: Identification of the solution}\quad\\
 We denote $P:=\frac{\eta_c}{b(1-p)}$ and for $i\geq1$, let $E_i$ be a non-negative random variable such that $\p({E_i}=0)=1-P$ and conditional on being positive, ${E_i}$ is exponential with parameter $\theta_i:=\left(\frac{1}{i!}\left(\frac{p}{1-p}\right)^i\frac{\eta_c^i}{\psi'_c(\eta_c)^{i+1}}\right)^{-1}.$
 To complete the proof of Theorem \ref{CV_types}, it remains to check that the Laplace transform of $E_i$, that is
   $\varphi_{E_i}(a)=1-P+P\frac{\theta_i}{a+\theta_i},$  satisfies the conditions \eqref{indoc} and \eqref{equation_carac_kappa_i} of the Step 2.
   
   Condition \eqref{indoc} is trivially satisfied since $\EE{E_i}=P/\theta_i$. Furthermore, we have
   {\setlength\arraycolsep{2pt}
\begin{eqnarray*}
\int_0^z\varphi_{E_i}(e^{-\eta_cu})\dif u&=&z\left(1-P\right)+P\int_0^z\frac{\theta_i e^{\eta_cu}}{a+\theta_i e^{\eta_cu}}\dif u\\
&=&z\left(1-P\right)+\frac{P}{\eta_c}\left(\ln(a+\theta_i e^{\eta_cz})-\ln(a+\theta_i)\right).
\end{eqnarray*}}
Then,
 {\setlength\arraycolsep{2pt}
\begin{eqnarray*}
\intpos\frac{\Lambda(\dif z)}{b}\exp\left\{b(1-p)\left(\int_0^z\varphi_{E_i}(ae^{-\eta_cu})\dif u-z\right)\right\}&=&\intpos\frac{\Lambda(\dif z)}{b}e^{-\eta_cz}\frac{a+\theta_i e^{\eta_cz}}{a+\theta_i}\\
&=&\frac{\theta_i}{a+\theta_i}+\frac{a}{a+\theta_i}\intpos\frac{\Lambda(\dif z)}{b}e^{-\eta_cz}.
   \end{eqnarray*}}
Since $\eta_c$ is a root of the function $\psi_c$ defined by \eqref{def_psi_c}, we have $(1-p)\intpos\Lambda(\dif z)e^{-\eta_cz}=b(1-p)-\eta_c$.
Then, $$\intpos\frac{\Lambda(\dif z)}{b}\exp\left\{b(1-p)\left(\int_0^z\varphi_{E_i}(ae^{-\eta_cu})\dif u-z\right)\right\}=\frac{\theta_i}{a+\theta_i}+\frac{a}{a+\theta_i}(1-P)=\varphi_{E_i}(a).$$
Thus, $\varphi_{E_i}$ satisfies condition \eqref{equation_carac_kappa_i}, which ends the proof of Theorem \ref{CV_types}.
\end{proof}

%\paragraph{Step 4: Characterization of $\varphi_i$}\quad\\

% We prove here that if a non-negative random variable has a positive and finite mean $\mu:=\EE{Y}$ and if its Laplace transform satisfies \eqref{equation_carac_kappa_i}, then
% $$\p(Y>0)=\frac{\eta_c}{b(1-p)},$$ and conditional on being positive, $Y$ is exponential with mean $$\mu\frac{b(1-p)}{\eta_c}.$$
% %
% If we show this, the proof of Theorem \ref{CV_types} is complete. 

 \begin{proof}[Proof of Lemma \ref{unicite}]
The proof of this result is mostly inspired from that of Theorem 7.2 in \cite{Doney1972}.
Let $\nu$ be the positive measure on $[0,+\infty)$ defined by $\nu(\dif x):=(1-p)e^{-\eta_cx}\dif x\int_x^{+\infty}\Lambda(\dif z)$. This is a probability measure since by the Fubini theorem, 
$$\intpos \nu(\dif x)=(1-p)\intpos \Lambda(\dif z)\int_0^ze^{-\eta_cx}\dif x=(1-p)\intpos \Lambda(\dif z)\frac{1-e^{-\eta_cz}}{\eta_c}=\frac{\eta_c-\psi_c(\eta_c)}{\eta_c}=1$$
where we have used that $\eta_c$ is the largest root of $\psi_c$ .

 Let $f$ and $g$ be two functions satisfying the three conditions of Lemma \ref{unicite}. For $a>0$, we define $\Delta(a):=\frac{|f(a)-g(a)|}{a}$. Since $f$ and $g$ fulfill the conditions (i) and (ii), we have 
 $$\frac{f(a)-g(a)}{a}=\frac{f(a)-1}{a}+\frac{1-g(a)}{a}\underset{a\to0}\longrightarrow f'(0)-g'(0)=0.$$
Hence, $\Delta(0^+)=0$ and since $f,g$ are continuous and bounded on $[0,+\infty)$, the function $\Delta$ also is. 
Moreover, since $f$ and $g$ satisfy (iii) and take values in $[0,1]$, using the inequality $|e^{-x}-e^{-y}|\leq |x-y|$ for $x,y\geq0$ , we have
\begin{align*}
\Delta(a)&\leq (1-p)\intpos \Lambda(\dif z)\int_0^z\Delta\left(ae^{-\eta_cu}
\right)e^{-\eta_cu}\dif u\\
&=(1-p)\intpos \dif u\Delta\left(ae^{-\eta_cu}\right)e^{-\eta_cu}\int_u^{\infty}\Lambda(\dif z)\\
&=\EE{\Delta\left(ae^{-\eta_cZ}\right)}
\end{align*}
where $Z$ is a random variable distributed as the probability measure $\nu$. Thus, iterating this argument,
\begin{equation}\label{creep}
\Delta(a)\leq \EE{\Delta\left(ae^{-\eta_c(Z_1+\cdots+Z_n)}\right)}
\end{equation}
where the random variables $Z_i$'s are i.i.d. and distributed as $\nu$. Using the definition of $\nu$, these random variables are a.s. positive. Thus, % and the fact that $\Lambda$ has a finite second moment, simple computations lead to %$0<\EE{Z}\leq (1-p)\sigma^2/2<\infty$. Hence, the strong law of large numbers ensures that 
the sum $Z_1+\cdots+Z_n$ a.s. goes to $+\infty$ as $n\to\infty$ and the r.h.s of \eqref{creep} goes to $\Delta(0^+)=0$ as $n\to\infty$ by the dominated convergence theorem (recall that $\Delta$ is bounded). Finally, $\Delta(a)=0$ for any positive $a$ and the proof is complete.
\end{proof}  

 \begin{rem}
We can also view $(L^i(t),i\geq0,t\geq0)$ as a multitype CMJ-process where the individuals of this branching process are the different alleles; roughly speaking, all the individuals sharing the same alleles are gathered into clusters. Moreover, an allele of type $i$
has a lifespan distributed as the extinction time of $\Xi_c$ and gives birth to alleles of type $i+1$ according to the jump times of $\left(Y_{\int_0^t\Xi(s)\dif s},t\geq0\right)$ where $Y$ is a homogeneous Poisson process with rate $bp$, independent of $\Xi_c$. 
 
In the supercritical case, it seems not possible to study the asymptotic behavior of $L_i(t)$ with the techniques we used to obtain that of $K_i(t)$. 
The problem is that in this multitype process, the number of ``children'' of an allele may be infinite since it survives forever with positive probability $\frac{\eta_c}{b(1-p)}$.
Then, some conditions of \cite[Thm 4.3]{Mode1971} are not fulfilled. As an example, 
the second moment property we checked in the first step of the proof of Theorem \ref{CV_types} is false here.
%the mean number of mutant alleles that are produced by an allele is infinite since each allele has a positive probability $\frac{\eta_c}{b(1-p)}$ to survive forever and then to give birth to an infinite number of children. 
 \end{rem}

\appendix
 %\section{Appendix}

 \section{Asymptotic behaviors of \texorpdfstring{$W$}{W} and \texorpdfstring{$W'$}{W'}}
 \subsection{Proof of Lemma \ref{W_et_W'}(i)}
We suppose here that $m>1$ which implies that $\eta>0$.
 To obtain the asymptotic behavior of $W$, we use a Tauberian theorem about Laplace transforms.
Indeed, we have $$\intpos e^{-\lambda x}\left(e^{-\eta x}W(t)\right)\dif x=\frac{1}{\psi(\eta+\lambda)}\underset{\lambda\to0}\sim\frac{1}{\lambda\pse}.$$
Then, since $t\longmapsto e^{-\eta t}W(t)$ is non-decreasing according to equation (4) in \cite{Chan_Kyprianou} (applied with $q=0$), a Tauberian theorem entails the desired result (see for instance \cite[p.10]{Levy_processes}).

To get the behavior of $W'$, since $t\longmapsto e^{-\eta t}W'(t)$ is not necessarily non-decreasing, we cannot use the same method. However, according to \eqref{formule_derivee_W},
$$e^{-\eta t}W'(t)=be^{-\eta t}W(t)-\int_0^te^{-\eta(t-x)}W(t-x)e^{-\eta x}\Lambda(\dif x).$$
Since $t\longmapsto e^{-\eta t}W(t)$ is non-decreasing and converges to $1/\pse$, the monotone convergence theorem implies that
$$\lim_{t\to\infty}e^{-\eta t}W'(t)=\frac{b}{\pse}-\frac{1}{\pse}\intpos e^{-\eta x}\Lambda(\dif x).$$
The proof is then complete since from \eqref{def_psi}, we have $0=\psi(\eta)=\eta-b+\intpos e^{-\eta x}\Lambda(\dif x)$.
 
\subsection{Proof of Lemma \ref{W_et_W'}(ii)}
We assume here that $m=1$ and $\sigma^2=\intpos z^2\Lambda(\dif z)<\infty$. 
First, notice that in that case $\psi'(0)=1-m=0$ and that $\psi''(0)=\sigma^2$. So 
$$\intpos e^{-\lambda x}W(x)\dif x=\frac{1}{\psi(\lambda)}\underset{\lambda\to0}\sim\frac{2}{\sigma^2\lambda^2}.$$ 
Hence, a Tauberian theorem yields the asymptotic behavior of $W(t)$ as $t\to\infty$.

As in the supercritical case, it is not possible to obtain the asymptotic behavior of $W'$ with the help of a Tauberian theorem because this function is not monotone. We use known results about the extinction time of a critical CMJ-process. 
According to the main result of \cite{Holte1974303}, recalling that the measure $\mu$ is defined by \eqref{defi_mu} and that $\zeta$ is the lifespan of the ancestor, if 
\begin{gather}
\lim_{t\to\infty}t^2(1-\mu([0,t])=0\label{cond_1}\\
\lim_{t\to\infty}t^2\p(\zeta>t)=0,\label{cond_2}
\end{gather}
then 
$$\lim_{t\to\infty}t\p(\Xi(t)>0)=\frac{2\intpos r\mu(\dif r)}{V}$$
where $V$ is the variance of the number of children of the ancestor.
Conditional on ``the ancestor has a lifespan $z$'', this number is Poisson with parameter $bz$. Then, we have $V=b\intpos r^2\Lambda(\dif r)$. Moreover, according to \eqref{defi_mu},
$$\intpos r\mu(\dif r)=\intpos r \dif r\int_r^\infty\Lambda(\dif x)=\frac{\sigma^2}{2}.$$
Thus,
\begin{equation}\label{comportementprobaextinction}\p(\Xi(t)>0)\underset{t\to\infty}\sim\frac{1}{tb}.
 \end{equation}
It remains to prove that the conditions \eqref{cond_1} and \eqref{cond_2} hold. By using \eqref{defi_mu}, \eqref{cond_1} is equivalent to the condition \eqref{cond_W_critique} that we have assumed here. Moreover, $\p(\zeta>t)=b^{-1}\Lambda((t,+\infty])$, which goes to $0$ faster than $t^{-2}$ as $t\to\infty$ since the second moment of $\Lambda$ is finite.
 
 On the other hand, according to \eqref{loi_Xt}, we know that $\p(\Xi(t)>0)=\frac{W'(t)}{bW(t)}$. Hence, $W'(t)$ converges to $2/\sigma^2$ as a consequence of \eqref{comportementprobaextinction} and of the fact that $W(t)$ behaves as $2/(\sigma^2t)$ as $t\to\infty$.

\subsection{Proof of Lemma \ref{W_et_W'}(iii)} 
As in the critical case, the asymptotic behavior of $W$ can be obtained via a Tauberian theorem because $\psi(\lambda)\sim(1-m)\lambda$ as $\lambda\to0$.

To study $W'$ in that case, we use again known results about the time of extinction of CMJ-processes.
The hypotheses \eqref{condJagers3} that we make about $\tilde \eta_c$ enables us to use Theorem 6.7.10 in \cite{Jagers_BP_with_bio}, that is,
$$\widetilde C:=\lim_{t\to\infty}e^{-\tilde \eta t}\p(\Xi(t)>0)$$ exists and is positive if and only if $$\Esp\left[\intpos e^{-\tilde \eta  t}\xi(\dif t)\log\{\xi((0,\infty))\}\right]<\infty$$
where we recall that $\xi$ is the birth point process of the ancestor.
By conditioning on the lifespan of the ancestor and by using classical properties about Poisson point processes,
we have
$$\Esp\left[\intpos e^{-\tilde \eta  t}\xi(\dif t)\log\{\xi((0,\infty))\}\right]\leq\int_{(0,\infty)}\Lambda(\dif z)\frac{1-e^{-\tilde \eta z}}{\tilde \eta }\left((1-p)+b(1-p)^2z\right)$$
which is finite thanks to \eqref{condJagers3}.
Then, we have proved that there exists $\widetilde C >0$ such that
\begin{equation}\label{afterf}\p(\Xi(t)>0)\underset{t\to\infty}\sim \widetilde C e^{\tilde \eta t}.
\end{equation}
It remains to compute $\widetilde C$ (which is unknown in most CMJ-processes). 
Since $W(t)$ converges to $(1-m)^{-1}$ as $t\to\infty$, using \eqref{loi_Xt} and \eqref{afterf},
\begin{equation}\label{afterfo}
W' (t)\underset{t\to\infty}\sim\frac{b}{1-m}\widetilde C e^{\tilde \eta t}.
\end{equation}
Moreover, by integrating \eqref{defi_W} by parts, since $W(0)=1$ according to \cite[Lem 4.1]{Splitting_trees_Immig}, we have for $\lambda>\tilde \eta$
$$\intpos W' (u)e^{-\lambda u}\dif u=\frac{\lambda}{\psi(\lambda)}-1$$
with the convention $0/\psi(0)=1/(1-m)<\infty$.
Then,
$$\intpos \left(W' (u)e^{-\tilde \eta  u}\right)e^{-\lambda u}\dif u=\frac{\lambda+\tilde \eta }{\psi(\lambda+\tilde \eta )}\underset{\lambda\to0}\sim\frac{\tilde \eta }{\lambda\psi'(\tilde \eta )}.$$
On the other hand, using \cite[Thm. 1.7.6]{Bingham1989} with \eqref{afterfo},
we get $$\intpos \left(W' (u)e^{-\tilde \eta  u}\right)e^{-\lambda u}\dif u\underset{\lambda\to0}\sim\frac{b\widetilde C }{1-m}\frac{1}{\lambda},$$
which implies that $\widetilde C=\frac{\tilde\eta}{\psi(\tilde\eta)}$.

\section*{Acknowledgments}
This work was supported by project MANEGE ANR-09-BLAN-0215
(French national research agency). I want to thank Amaury Lambert for his very helpful remarks and comments and Guillaume Achaz who raised a question about the number of mutations undergone 
by an allele during a talk given at Coll\`ege de France. My thanks also to two anonymous referees for their careful check of this paper and for helpful remarks.

%\bibliographystyle{abbrv}
%\bibliography{/home/mathieu/Dropbox/Recherche/ma_biblio}

 \end{document}